\documentclass[a4paper,leqno,10pt]{amsart}

\usepackage{epsf,graphicx,epsfig}
\usepackage{amscd}
\usepackage{amsmath,latexsym,amssymb,amsthm}
\usepackage[nospace,noadjust]{cite}
\usepackage{textcomp}
\usepackage{setspace,cite}
\usepackage{lscape,fancyhdr,fancybox}
\usepackage{stmaryrd}
\usepackage[all,cmtip]{xy}
\usepackage{tikz}
\usepackage{cancel}
\usetikzlibrary{shapes,arrows,decorations.markings}
\usepackage{graphicx}

\usepackage{color}
\usepackage{url}
\usepackage{enumerate}
\usepackage[mathscr]{euscript}

  \theoremstyle{plain}

\swapnumbers
    \newtheorem{thm}{Theorem}[section]
    \newtheorem{prop}[thm]{Proposition}
     
   \newtheorem{lemma}[thm]{Lemma}

    \newtheorem{subsec}[thm]{}
\theoremstyle{definition}
    \newtheorem{defn}[thm]{Definition}
        \newtheorem{remark}[thm]{Remark}
    \newtheorem{exam}[thm]{Example}

\theoremstyle{remark}

\title{}
\author{}
\date{}
\usepackage{amssymb}

\usepackage{hyperref}
\hypersetup{
	colorlinks,
	citecolor=blue,
	filecolor=black,
	linkcolor=blue,
	urlcolor=black
}

\begin{document}
\title{Cohomology of BiHom-associative algebras}

\author{Apurba Das}
\address{Department of Mathematics and Statistics,
Indian Institute of Technology, Kanpur 208016, Uttar Pradesh, India.}
\email{apurbadas348@gmail.com}

\subjclass[2010]{16E40, 17A30.}
\keywords{Bihom-associative algebra, Hochschild cohomology, operad, Gerstenhaber algebra, formal deformation, Bihom-$A_\infty$-algebra.}

\begin{abstract}
Bihom-associative algebras have been recently introduced in the study of group hom-categories. In this paper, we introduce a Hochschild type cohomology for bihom-associative algebras with suitable coefficients. The underlying cochain complex (with coefficients in itself) can be given the structure of an operad with a multiplication. Hence, the cohomology inherits a Gerstenhaber structure. We show that this cohomology also control corresponding formal deformations. Finally, we introduce bihom-associative algebras up to homotopy and show that some particular classes of these homotopy algebras are related to the above Hochschild cohomology.
\end{abstract}

\noindent

\thispagestyle{empty}

\maketitle

\section{Introduction}
A pioneer work of Gerstenhaber shows that the Hochschild cohomology of an associative algebra controls the deformation of the given associative structure \cite{gers-def}. He also observes that the Hochschild cohomology inherits a rich structure which is now known as Gerstenhaber algebra \cite{gers}. After that, either or both of these results have been extended in other type of algebras as well.

Recently, hom-type algebras have been studied by many authors. In these algebras, the identities defining the structures are twisted by homomorphisms. The notion of hom-Lie algebras was first appeared in the $q$-deformations of the Witt and Virasoro algebras
\cite{hls}. Other types of algebras (e.g. associative, Leibniz, Poisson, Hopf,...) twisted by homomorphisms including their representations, cohomology and deformations have also been studied. A hom-associative algebra is an algebra $(A, \mu)$ whose associativity is twisted by an algebra homomorphism $\alpha$ in the sense that: $\mu (\alpha (a), \mu (b, c)) = \mu (\mu(a,b), \alpha (c))$, for all $a, b, c \in A$. See \cite{makh-sil, makh-sil2, amm-ej-makh, das1, das2} (and references there in) for more details.

In \cite{caen-goyv} the authors gave a new look to hom-type algebras from a category theoretical point of view. A generalization of this approach led the authors of \cite{gra-makh-men-pana} to introduce bihom-associative algebras. In these algebras, the associativity of the multiplication is twisted by two commuting homomorphisms (see Definition \ref{defn-bihom}). When these two homomorphisms are equal, one recovers hom-associative algebras.  Given an associative algebra $(A, \mu)$ and two commuting algebra morphisms $\alpha, \beta : A \rightarrow A$, the quadruple $(A, \mu \circ (\alpha \otimes \beta), \alpha, \beta)$ is a bihom-associative algebra. The authors in \cite{gra-makh-men-pana} also introduced the notion of a representation for bihom-associative algebras. 
However, the cohomology, structure of the cohomology (if such exists) and deformation of bihom-associative algebras have not yet studied. In this paper, we aim to answer this.

First we define the Hochschild cohomology of a bihom-associative algebra $A$ with coefficients in itself. Using it we define cohomology with coefficients in an arbitrary bimodule representation. The second Hochschild cohomology can be interpreted as equivalence classes of abelian extensions of bihom-associative algebras.
We show that the cochain complex defining the Hochschild cohomology of $A$ (with coefficients in itself) can be given the structure of an operad with a multiplication. The Hochschild coboundary can be seen (up to some sign) as the differential induced by that operad with multiplication. Thus, a classical work of Gerstenhaber and Voronov \cite{gers-voro} suggests that the Hochschild cochain complex carries a homotopy $G$-algebra and the cohomology inherits a Gerstenhaber structure.

In the next, we study formal $1$-parameter deformations of a bihom-associative algebra $A$. This discussion is similar to the classical deformation of associative algebras studied by Gerstenhaber \cite{gers-def}.  
The vanishing of the second cohomology implies that the algebra $A$ is rigid (i.e. all deformations of $A$ are equivalent to the trivial one) and the vanishing of the third cohomology allows one to extend any finite order deformation to a next order deformation.
 
Finally, we introduce a strongly homotopy version of bihom-associative algebras, which we call bihom-$A_\infty$-algebras. This definition is motivated by the classical notion of $A_\infty$-algebras introduced by Stasheff \cite{sta} and Hochschild cohomology of bihom-associative algebras introduced in this paper. 
We classify `skeletal' and `strict' bihom-$A_\infty$-algebras. Skeletal algebras are related to Hochschild cohomology of bihom-associative algebras and strict algebras correspond to crossed modules of bihom-associative algebras.

All vector spaces and linear maps are defined over a field $\mathbb{K}$ of characteristic zero.

\section{Preliminaries}\label{sec3}
In this section we recall bihom-associative algebras and some basics on operad \cite{gra-makh-men-pana, gers-voro, lod-val-book}.

\subsection{Bihom-associative algebras} Let us first recall some motivation from \cite{gra-makh-men-pana} to define bihom-associative algebras. Given a monoidal category $\mathcal{C}$, one consider a new monoidal category $\widetilde{\mathcal{H}} (\mathcal{C})$, called hom-category, whose objects are pairs consisting of an object in $\mathcal{C}$ and an automorphism of this object \cite{caen-goyv}. By considering $\mathcal{C} = {}_{\mathbb{K}}\mathcal{M}$, the category of $\mathbb{K}$-modules, one observes that an algebra in the monoidal category $\widetilde{\mathcal{H}} ({}_{\mathbb{K}}\mathcal{M})$ is a hom-associative algebra $(A, \mu, \alpha)$ with bijective structure map $\alpha$. In \cite{graz} the author extends this construction by adding a group action. More precisely, given a group $\mathcal{G}$, two elements $c, d \in Z (\mathcal{G})$ and an automorphism $\nu$ of the unit object in $\mathcal{C}$, the author constructs the group hom-category $\mathcal{H}^{c, d, \nu} (\mathcal{G}, \mathcal{C})$ whose objects are pairs of the form $(A, f_A)$, where $A$ is an object in $\mathcal{C}$ and $f_A : \mathcal{G} \rightarrow \text{Aut}_\mathcal{C} (A)$ is a group homomorphism. When $\mathcal{G} = \mathbb{Z}, c = d = 1_\mathbb{Z}$ and $\nu = \text{id}$, one recovers the category $\widetilde{\mathcal{H}} (\mathcal{C})$ considered in \cite{caen-goyv}. One may also consider the case $\mathcal{G} = \mathbb{Z} \times \mathbb{Z}, c = (1,0), d = (0,1)$ and $\nu = \text{id}.$ By writing down the axioms for an algebra in the category $\mathcal{H}^{(1,0), (0,1), 1} (\mathbb{Z} \times \mathbb{Z}, {}_{\mathbb{K}}\mathcal{M})$ and relaxing the invertibility of certain maps, one comes to the following notion of a bihom-associative algebra \cite{gra-makh-men-pana}.

\begin{defn}\label{defn-bihom}
	A bihom-associative algebra over $\mathbb{K}$ is a quadruple $(A, \mu, \alpha, \beta)$ consisting of a vector space $A$ together with a bilinear map
	$\mu : A \times A \rightarrow A$ and two commuting linear maps $\alpha, \beta : A \rightarrow A$ satisfying $\alpha (\mu (a,b)) = \mu(\alpha (a), \alpha (b))$, $\beta (\mu (a,b)) = \mu(\beta (a), \beta (b))$ and the following bihom-twisted associativity
	\begin{align}\label{hom-ass-cond}
	\mu ( \alpha (a) , \mu ( b , c) ) = \mu ( \mu (a , b) , \beta (c)), ~~ \text{ for all } a, b, c \in A.
	\end{align}
\end{defn}

When $\alpha = \beta$, one recovers the notion of a hom-associative algebra. When they both are identity map, one gets the definition of a classical associative algebra. Thus, they are very first examples of bihom-associative algebras.

\begin{exam}
Let $(A, \mu)$ be an associative algebra and $\alpha, \beta : A \rightarrow A$ be two commuting algebra morphisms. Then $(A, \mu \circ (\alpha \otimes \beta), \alpha, \beta)$ is a bihom-associative algebra, called `Yau twist'.

Any bihom-associative algebra $(A, \mu, \alpha, \beta)$ with $\alpha$ and $\beta$ are bijective arises in this way.
\end{exam}

\begin{exam}
Let $A = (A, \mu, \alpha, \beta)$ and $A' = (A' , \mu', \alpha', \beta')$ be two bihom-associative algebras. Then $A \otimes A' = (A \otimes A', \mu \otimes \mu', \alpha \otimes \alpha', \beta \otimes \beta')$ is a bihom-associative algebra, called the tensor product of $A$ and $A'$.
\end{exam}

\begin{exam}
Let $ (A, \mu, \alpha, \beta)$ be a bihom-associative algebra. Then $(M_n (A), M_n (\mu), M_n (\alpha), M_n (\beta))$ is also a bihom-associative algebra whose multiplication $M_n (\mu)$ is given by the matrix multiplication induced by $\mu$ and the structure map $M_n (\alpha)$ (resp. $M_n (\beta)$) is the map $\alpha$ (resp. $\beta$) on each entry of the matrix.
\end{exam}

 Let $A = \mathbb{K}[X]$ be the algebra of polynomials in $X$ and consider the algebra morphisms $\alpha, \beta : A \rightarrow A$ defined by the equalities $\alpha (X) = X^2$ and $\beta = \text{id}_A$. Then it is shown in \cite{gra-makh-men-pana} that the bihom-associative algebra $(A, \mu \circ (\alpha \otimes \beta) , \alpha, \beta)$ cannot be expressed as a hom-associative algebra. Thus, the category of bihom-associative algebras is larger than the category of hom-associative algebras.

See \cite{gra-makh-men-pana} for some more examples.

A morphism $\phi : (A, \mu, \alpha, \beta) \rightarrow (A', \mu', \alpha', \beta')$ between bihom-associative algebras is a linear map $\phi: A \rightarrow A'$ which commutes with all respective structure maps.

\begin{remark}
Let $(A, \mu, \alpha, \beta)$ be a bihom-associative algebra. If we set $a \cdot b = \mu (a, b)$, for $a, b \in A$, then for any $n \geq 3$ and $a_1, \ldots , a_n \in A$, we have
\begin{align*}
\alpha^{n-2} a_1 \cdot \big(  \cdots (\alpha^2 a_{n-3} \cdot ( \alpha a_{n-2} \cdot (a_{n-1} \cdot a_n))) \cdots \big) = ( \cdots ((( a_1 \cdot a_2 ) \cdot  \beta a_3 ) \cdot \beta^2 a_4) \cdots ) \cdot \beta^{n-2} a_n.
\end{align*}
This is clearly true for $n = 3$. Using induction, one can show that it holds for any $n \geq 3$.
\end{remark}

\begin{defn} A left representation of a bihom-associative algebra $A = (A, \mu, \alpha, \beta)$ consists of a quadruple $M = (M, \mu_l, \alpha_M, \beta_M)$ of a vector space $M$, a bilinear map $\mu_l : A \otimes M \rightarrow M$, $a \otimes m \mapsto a \cdot m$ and two commuting linear  maps $\alpha_M , \beta_M : M \rightarrow M$ satisfying
\begin{align*}
\alpha_M (a \cdot m) = \alpha (a) \cdot \alpha_M (m), ~~~ \beta_M (a \cdot m) = \beta (a) \cdot \beta_M (m) ~~~ \text{ and } ~~~ \alpha (a) \cdot (b \cdot m) = \mu (a, b) \cdot \beta_M (m).
\end{align*}
\end{defn}
A right representation $M = (M, \mu_r, \alpha_M, \beta_M)$ can be defined in a similar way. A bimodule is a quintuple $M = (M, \mu_l, \mu_r, \alpha_M, \beta_M)$ such that $(M, \mu_l, \alpha_M, \beta_M)$ is a left representation and $(M, \mu_r, \alpha_M, \beta_M)$ is a right representation and the following equality holds:
\begin{align*}
\alpha(a) \cdot (m \cdot b) = (a \cdot m) \cdot \beta (b).
\end{align*}

Any bihom-associative algebra $A = (A, \mu, \alpha, \beta)$ is a bimodule over itself with $\mu_l = \mu_r = \mu$ and $\alpha_M = \alpha$, $\beta_M = \beta$. When there is an $A$-bimodule $M = (M, \mu_l, \mu_r, \alpha_M, \beta_M)$, we denote both $\mu_l$ and $\mu_r$ by $\mu$ if there is no confusion and simply say $M = (M, \mu, \alpha_M, \beta_M)$ is an $A$-bimodule. Sometimes we will also use $\cdot$ (dot) to denote either of the left or right action of $A$ on $M$.

A morphism between $A$-bimodules $M$ and $N$ is given by a linear map $\psi : M \rightarrow N$ which commutes with respective structure maps of $M$ and $N$, and also preserves left and right $A$ actions.

If $M = (M, \mu, \alpha_M, \beta_M)$ is an $A$-bimodule, one can define a bihom-associative algebra structure on $M \oplus A$. The product and structure maps are given by
\begin{align*}
\overline{\mu} ((m, a), (n, b)) =~& (m \cdot b + a \cdot n,~ \mu (a, b)),\\
\overline{\alpha}((m, a)) = (\alpha_M (a), \alpha (a)) ~~~ &\text{~~~ and ~~~} ~~~ \overline{\beta} ((m, a)) = (\beta_M (m), \beta (a)).
\end{align*}
This is called the semi-direct product bihom-associative algebra and is denoted by $M \rtimes A.$

\subsection{Operads with multiplication}\label{subsec-operad-mul} Here we recall some necessary backgroud on non-symmetric operads \cite{gers-voro}. For more details on (non-symmetric) operads, see \cite{lod-val-book}.
\begin{defn}
	A  non-symmetric operad (non-$\sum$ operad in short) in the category of vector spaces is a collection of vector spaces $\{ \mathcal{O} (k) |~ k \geq 1 \}$ together with compositions
	\begin{align*}
	\gamma : \mathcal{O} (k) \otimes \mathcal{O} (n_1) \otimes \cdots \otimes \mathcal{O} (n_k) \rightarrow \mathcal{O} (n_1 + \cdots + n_k) ,~~
	f \otimes g_1 \otimes \cdots \otimes g_k  \mapsto  \gamma (f ; g_1 , \ldots, g_k )
	\end{align*}
	 which is associative in the sense that 
\begin{align*}
\gamma \big(&  \gamma (f; g_1 , \ldots, g_k); h_1 , \ldots, h_{n_1 + \cdots + n_k}    \big) \\
&= \gamma \big( f; ~\gamma (g_1 ; h_1, \ldots, h_{n_1}), ~ \gamma ( g_2 ; h_{n_1 + 1}, \ldots, h_{n_1 + n_2}) , \ldots, ~\gamma (g_k ; h_{n_1 + \cdots + n_{k-1}+1}, \ldots, h_{n_1 + \cdots + n_k})  \big)
\end{align*}	 
and there is an identity element id$ \in \mathcal{O} (1)$ such that
$\gamma (f ; \underbrace{\text{id}, \ldots, \text{id}}_{k \text{ times}} ) = f = \gamma (\text{id}; f)$, for  $f \in \mathcal{O} (k)$. 
\end{defn}

A non-$\sum$ operad can also be described by partial compositions
$$\circ_i : \mathcal{O}(m) \otimes \mathcal{O}(n) \rightarrow \mathcal{O}(m+n-1), \quad 1 \leq i \leq m$$
satisfying
$$ \begin{cases} (f \circ_i g) \circ_{i+j-1} h = f \circ_i (g \circ_j h), \quad &\mbox{~~~ for } 1 \leq i \leq m, ~1 \leq j \leq n, \\ (f \circ_i g) \circ_{j+n-1} h = (f \circ_j h) \circ_i g, \quad  & \mbox{~~~ for } 1 \leq i < j \leq m, \end{cases}$$
for $f \in \mathcal{O}(m), ~ g \in \mathcal{O}(n), ~ h \in \mathcal{O}(p),$
and an identity element satisfying $ f \circ_i \text{id} = f =\text{id}\circ_1 f,$ for all $f \in \mathcal{O}(m)$ and $1 \leq i \leq m$. The two definitions of non-$\sum$ operad are related by
\begin{align}
f \circ_i g =~& \gamma (f ;~ \overbrace{\text{id}, \ldots, \text{id} , \underbrace{g}_{i\text{-th place}}, \text{id}, \ldots, \text{id}}^{m\text{-tuple}}),~~~ \quad \text{ for }f \in \mathcal{O}(m),\label{eqn-1}\\
\gamma (f ; g_1, \ldots, g_k) =~&   (\cdots ((f \circ_k g_k) \circ_{k-1} g_{k-1}) \cdots ) \circ_1 g_1 , \quad \text{ for }f \in \mathcal{O}(k). \label{eqn-2}
\end{align}

\medskip

Next, consider the graded vector space $\mathcal{O} = \oplus_{k \geq 1} \mathcal{O}(k)$ of an operad. If $f \in \mathcal{O}(n)$, we define
 $|f| = n-1$. We will use the same notation for any graded vector space as well. Consider the braces
\begin{align}\label{gers-voro-brace}
\{ f \} \{ g_1, \ldots, g_n\} := \sum (-1)^\epsilon ~ \gamma (f ; \text{id}, \ldots, \text{id}, g_1, \text{id}, \ldots, \text{id}, g_n, \text{id}, \ldots, \text{id})
\end{align}
where the summation runs over all possible substitutions of $g_1, \ldots, g_n$ into $f$ in the prescribed order and $\epsilon := \sum_{p=1}^{n} |g_p| i_p$. Here $i_p$ is the total number of inputs in front of $g_p$. The multilinear braces $\{f\} \{ g_1, \ldots, g_n\}$ are homogeneous of degree $-n$. 
We use the conventions that
$\{f\}\{~ \} := f$   and $f \circ g := \{f\} \{ g\}.$
The braces (\ref{gers-voro-brace}) satisfy certain higher pre-Lie identities which in particular imply that
\begin{align}\label{pre-lie-iden}
(f \circ g) \circ h - f \circ (g \circ h) = (-1)^{|g||h|} ((f \circ h) \circ g - f \circ (h \circ g)).
\end{align}
This implies that the bracket
\begin{align}\label{lie-brckt}
[f,g] := f \circ g - (-1)^{|f||g|} g \circ f, ~~~ \text{ for } f , g \in \mathcal{O},
\end{align}
defines a degree $-1$ graded Lie bracket on $\mathcal{O}$.

\begin{defn}
	A multiplication on an operad $\mathcal{O}$ is an element $m \in \mathcal{O} (2)$ such that $m \circ m = 0$.
\end{defn}

If $m$ is a multiplication on an operad $\mathcal{O}$, then the dot product
\begin{center}
$f \cdot g = (-1)^{|f| + 1} \{ m \} \{ f, g\}, ~~~ f , g \in \mathcal{O},$
\end{center}
defines a graded associative algebra structure on $\mathcal{O}$. Moreover, the degree one map $d : \mathcal{O} \rightarrow \mathcal{O}$, $f \mapsto m \circ f - (-1)^{|f|} f \circ m$ is a differential on $\mathcal{O}$ and the triple $(\mathcal{O}, \cdot, d)$ is a differential graded associative algebra. Summarizing the properties of braces on operad, dot product and differential induced by a multiplication, one gets a homotopy $G$-algebra $(\mathcal{O}, \{~\} \{~, \ldots, ~\}, \cdot, d)$. See \cite{gers-voro} for a precise definition of a homotopy $G$-algebra.

As a summary, we get the following.
\begin{thm}\label{gers-voro-thm}
A multiplication on an operad $\mathcal{O}$ defines the structure of a homotopy $G$-algebra on $\mathcal{O} = \oplus \mathcal{O}(n).$
\end{thm}

Next, we recall Gerstenhaber algebras ($G$-algebras in short).

\begin{defn}
A Gerstenhaber algebra is a graded commutative associative algebra $(\mathcal{A} = \oplus \mathcal{A}^i, \cdot)$ together with a degree $-1$ graded Lie bracket $[-,-]$ on $\mathcal{A}$ satisfying the following graded Leibniz rule
$$[a, b \cdot c] = [a, b] \cdot c + (-1)^{|a|(|b| + 1)} b \cdot [a, c],~~~ \text{ for homogeneous } a, b, c \in \mathcal{A}.$$
\end{defn}

\begin{remark}\label{gers-voro-rem}
 Given any homotopy $G$-algebra $(\mathcal{O}, \{~\} \{~, \ldots, ~\}, \cdot, d)$, the product $\cdot$ induces a graded commutative associative product on the cohomology $H^\bullet (\mathcal{O}, d)$. The degree $-1$ graded Lie bracket as defined in (\ref{lie-brckt}) also passes on to the cohomology $H^\bullet (\mathcal{O}, d)$. Moreover, the induced product and the bracket on the cohomology satisfy the graded Leibniz rule to become a Gerstenhaber algebra \cite{gers-voro}.
\end{remark}

\section{Cohomology}\label{sec-cohomology}
 In this section, we introduce Hochschild cohomology of bihom-associative algebras.
We show that second Hochschild cohomology can be interpreted as equivalence classes of abelian extensions.

\subsection{Cohomology with self coefficients.}\label{subsec-cohomo-self}
Let $(A, \mu, \alpha, \beta)$ be a bihom-associative algebra. For each $n \geq 1$, we define a vector space $C^n_{\text{Hoch}} (A,A) $ consisting of all multilinear maps $f : A^{\otimes n} \rightarrow A$ satisfying $\alpha \circ f = f \circ \alpha^{\otimes n}$ and $\beta \circ f = f \circ \beta^{\otimes n}$.
Define a map $\delta_{\text{Hoch}} : C^n_{\mathrm{Hoch}} (A, A) \rightarrow C^{n+1}_{\mathrm{Hoch}} (A, A)$ by
\begin{align}\label{bihom-diff}
(\delta_{\mathrm{Hoch}} f) (a_1,  \ldots, a_{n+1}) =~& \mu \big(\alpha^{n-1}(a_1) , f(a_2, \ldots , a_{n+1}) \big) \\
 +&  \sum_{i=1}^{n} (-1)^i~ f \big( \alpha(a_1), \ldots, \alpha (a_{i-1}), \mu (a_i , a_{i+1}), \beta (a_{i+2}), \ldots, \beta (a_{n+1}) \big) \nonumber \\
 +& (-1)^{n+1} \mu (f (a_1, \ldots, a_n) , \beta^{n-1} (a_{n+1})), \nonumber
\end{align}
for $a_1, a_2, \ldots, a_{n+1} \in A$.
 The linear map $\alpha$ appears before the term $\mu (a_i, a_{i+1})$ and $\beta$ appears onward. Moreover, in the first term we have the power of $\alpha$ and in the last term we have the power of $\beta$. This is not a guess as it can seem because the differential is
induced by an operad with a multiplication. 
Moreover, this technique avoids a long computation to show $(\delta_{\mathrm{Hoch}})^2 = 0$. Please see subsection \ref{cohomo-bihom} for detailed clarification. (A similar description for the cohomology of hom-associative algebras can be found in \cite{das2}.)  The cohomology of this complex is called the Hochschild cohomology of the bihom-associative algebra $(A, \mu, \alpha, \beta)$ and denoted by $H^\bullet_{\mathrm{Hoch}} (A, A).$ 

When $\alpha = \beta$, one recovers the Hochschild cohomology of hom-associative algebras \cite{amm-ej-makh}. In particular, for $\alpha = \beta = \text{id}$, we get the classical Hochschild cohomology of associative algebras.

\subsection{Cohomology with coefficients in a bimodule} Using the cohomology with self coefficients, we now define cohomology with coefficients in a bimodule.

Let $A = (A, \mu, \alpha, \beta)$ be a bihom-associative algebra and $M = (M, \mu, \alpha_M, \beta_M)$ an $A$-bimodule. The $n$-th Hochschild cochain group of $A$ with coefficients in $M$ is given by
\begin{align*}
C^n_{\mathrm{Hoch}} (A, M) = \big\{ f : A^{\otimes n} \rightarrow M |~ \alpha_M \circ f = f \circ \alpha^{\otimes n} ~\text{~ and ~}~ \beta_M \circ f = f \circ \beta^{\otimes n} \big\}.
\end{align*}
Any map $f \in C^n_{\mathrm{Hoch}} (A, M)$ can be extended to a map $\widetilde{f} \in C^n_{\mathrm{Hoch}}(M \rtimes A , M \rtimes A)$ by
\begin{align*}
\widetilde{f} \big( (m_1, a_1), \ldots, (m_n, a_n) \big) = \big(  f (a_1, \ldots, a_n), 0 \big).
\end{align*}
The map $f$ can be obtained from $\widetilde{f}$ just by restricting it to $A^{\otimes n}.$ Moreover, $\widetilde{f} = 0$ implies that $f = 0$. Observe that the differential $\delta_{\mathrm{Hoch}} (\widetilde{f})$  has the property that it takes $A^{\otimes (n+1)}$ to $M$. We define the coboundary map $\delta_{\mathrm{Hoch}} : C^n_{\mathrm{Hoch}} (A, M) \rightarrow C^{n+1}_{\mathrm{Hoch}} (A, M)$ by
\begin{align*}
\delta_{\mathrm{Hoch}} (f) = (\delta_{\mathrm{Hoch}} (\widetilde{f}) )|_{A^{\otimes (n+1)}}. 
\end{align*}
One can easily observe that $\widetilde{\delta_{\mathrm{Hoch}} (f)} = \delta_{\mathrm{Hoch}} (\widetilde{f}).$ This shows that
\begin{align*}
\widetilde{(\delta_{\mathrm{Hoch}})^2 (f)} = \delta_{\mathrm{Hoch}} ( \widetilde{\delta_{\mathrm{Hoch}} (f)} ) = ({\delta_{\mathrm{Hoch}})^2 (\widetilde{f})} = 0.
\end{align*}
Hence $(\delta_{\mathrm{Hoch}})^2 = 0$. Note that in above we use $\delta_{\mathrm{Hoch}}$ as both the Hochschild coboundary of $A$ with coefficients in $M$ and the Hochschild coboundary of the semi-direct product $M \rtimes A$ with self coefficients. Thus make sure that there is no confusion.
Note that, if we write the expression for $\delta_{\mathrm{Hoch}} (f)$, it is the same as (\ref{bihom-diff}) where 
in the first and last terms $\mu$ denotes
the left and right action of $A$ on $M$, respectively. One could define the coboundary operator in this way and directly verify that it is a square zero map. However, our technique avoids some long computations.

The cohomology of the complex $(C^\bullet_{\mathrm{Hoch}} (A, M) , \delta_{\mathrm{Hoch}})$ is called the Hochschild cohomology of $A$ with coefficients in the $A$-bimodule $M$, and denoted by $H^\bullet_{\mathrm{Hoch}} (A, M).$

When $M = A$ with the obvious bimodule structure, the cohomology coincides with the one defined in subsection  \ref{subsec-cohomo-self}.

Next we observe some basic properties of the Hochschild cohomology of bihom-associative algebras which are similar to the classical case.

\begin{prop}
Let $\phi : A \rightarrow A'$ be a bihom-associative algebra morphism and $M$ an $A'$-bimodule. Then $\phi$ induces an $A$-bimodule structure on $M$ (denoted by $\phi^*M$) and
\begin{align*}
\phi^* : C^n_{\mathrm{Hoch}} (A', M) \rightarrow C^n_{\mathrm{Hoch}} (A, \phi^* M), ~f \mapsto f \circ \phi^{\otimes n} 
\end{align*}
induces a map $\phi^* : H^\bullet_{\mathrm{Hoch}} (A', M) \rightarrow H^\bullet_{\mathrm{Hoch}} (A, \phi^* M).$
\end{prop}

\begin{prop}
Let $A$ be a bihom-associative algebra and $\psi : M \rightarrow N$ be a morphism between $A$-bimodules. Then 
\begin{align*}
\psi_* : C^n_{\mathrm{Hoch}} (A, M) \rightarrow C^n_{\mathrm{Hoch}} (A, N), ~ f \mapsto \psi \circ f
\end{align*}
induces a map $\psi_* : H^\bullet_{\mathrm{Hoch}} (A, M) \rightarrow H^\bullet_{\mathrm{Hoch}} (A, N)$. In fact, $H^\bullet_{\mathrm{Hoch}} (A, -)$ is a functor from the category of $A$-bimodules to the category of $\mathbb{K}$-modules.
\end{prop}

\subsection{Abelian extensions.} In this subsection, we show that the second Hochschild cohomology $H^2_{\mathrm{Hoch}} (A, M)$ can be interpreted as equivalence classes of abelian extensions of bihom-associative algebras.

Let $A = (A, \mu, \alpha, \beta)$ be a bihom-associative algebra and $M= (M, \alpha_M, \beta_M)$ be a vector space equipped with two commuting linear maps. Note that $M$ can be considered as a bihom-associative algebra with trivial multiplication.

\begin{defn}
An abelian extension of $A$ by $M$ is an exact sequence of bihom-associative algebras
\[
\xymatrix{
0 \ar[r] &  (M, 0, \alpha_M, \beta_M) \ar[r]^{i} & (E, \mu_E, \alpha_E, \beta_E) \ar[r]^{j} & (A, \mu, \alpha, \beta) \ar[r] \ar@<+4pt>[l]^{s} & 0
}
\]
together with a splitting (given by $s$) which satisfies 
\begin{align}\label{s-property}
\alpha_E \circ s = s \circ \alpha ~~~\text{  and  }~~~ \beta_E \circ s = s \circ \beta.
\end{align}
\end{defn}

An abelian extension induces an $A$-bimodule structure on $(M, \alpha_M, \beta_M)$ via the actions $a \cdot m = \mu_E (s(a), i(m))$ and $m \cdot a = \mu_E (i(m), s(a))$, for $a \in A,~ m \in M.$ One can easily verify that this action is independent of the choice of $s$. 

\begin{remark}
Let $(E, \alpha_E, \beta_E)$ and $(A, \alpha, \beta)$ be two vector spaces equipped with commuting linear maps. Suppose $j : E \rightarrow A$ is a linear surjective map commuting with respective structure maps. Then there might not be a section $s : A \rightarrow E$ of $j$ which commutes with respective structure maps. Take $E = \langle x , y \rangle$ with $\alpha_E (x) = y,~ \alpha_E (y) = 0,~ \beta_E = \text{id}$ and $A = \langle a \rangle$ with $\alpha (a) = 0,~ \beta = \text{id}.$ Take $j (x) = a$ and $j (y) = 0$. Let $s$ be a section for $j$ commuting with respective structure maps. For $s (a) = \lambda x + \nu y$, we have $a = (j \circ s) (a) = \lambda a$, which implies that $\lambda = 1$. Finally,
\begin{align*}
0 = (s \circ \alpha) (a) = (\alpha_E \circ s)(a) = \alpha_E (x + \nu y) = y
\end{align*}
which is a contradiction.
\end{remark}

Two abelian extension are said to be equivalent if there is a morphism $\phi : E \rightarrow E'$ between bihom-associative algebras making the following diagram commute
\[
\xymatrix{
0 \ar[r] &  (M, 0, \alpha_M, \beta_M) \ar[r]^{i} \ar@{=}[d] & (E, \mu_E, \alpha_E, \beta_E) \ar[d]^{\phi} \ar[r]^{j} & (A, \mu, \alpha, \beta) \ar[r] \ar@{=}[d] \ar@<+4pt>[l]^{s} & 0 \\
0 \ar[r] &  (M, 0, \alpha_M, \beta_M) \ar[r]^{i'} & (E', \mu'_E, \alpha'_E, \beta'_E) \ar[r]^{j'} & (A, \mu, \alpha, \beta) \ar[r] \ar@<+4pt>[l]^{s'} & 0 .
}
\]
Note that two extensions with same $i$ and $j$ but different $s$ are always equivalent.

Suppose $M$ is a given $A$-bimodule. We denote by $\mathcal{E}xt (A, M)$ the equivalence classes of abelian extensions of $A$ by $M$ for which the induced $A$-bimodule structure on $M$ is the prescribed one.

The next result is inspired by the classical case.
\begin{prop}
There is a canonical bijection: $H^2_{\mathrm{Hoch}} (A, M) \cong \mathcal{E}xt (A, M).$
\end{prop}

\begin{proof}
Given a $2$-cocycle $f \in C^2_{\mathrm{Hoch}} (A, M)$, we consider the $k$-module $E = M \oplus A$ with the following structure maps
\begin{align*}
{\mu}_E ((m, a), (n, b)) =~& (m \cdot b + a \cdot n + f (a, b),~ \mu (a, b)),\\
{\alpha}_E((m, a)) = (\alpha_M (a), \alpha (a)) ~~~ &\text{~~~ and ~~~} ~~~ {\beta}_E ((m, a)) = (\beta_M (m), \beta (a)).
\end{align*}
(Observe that when $f =0$ this is the semi-direct product.) Using the fact that $f$ is a $2$-cocycle, it is easy to verify that $(E, \mu_E, \alpha_E, \beta_E)$ is a bihom-associative algebra. Moreover, $0 \rightarrow M \rightarrow E \rightarrow A \rightarrow 0$ defines an abelian extension with the obvious splitting. Let $(E' = M \oplus A, \mu_E', \alpha_E, \beta_E)$ be the corresponding bihom-associative algebra associated to the cohomologous $2$-cocycle $f - \delta_{\mathrm{Hoch}} (g)$, for some $g \in C^1_{\mathrm{Hoch}} (A, M)$. The equivalence between abelian extensions $E$ and $E'$ is given by $E \rightarrow E'$, $(m, a) \mapsto (m + g (a), a)$. Therefore, the map $H^2_{\mathrm{Hoch}} (A, M) \rightarrow \mathcal{E}xt (A, M) $ is well defined.

Conversely, given an extension 
$0 \rightarrow M \xrightarrow{i} E \xrightarrow{j} A \rightarrow 0$ with splitting $s$, we may consider $E = M \oplus A$ and $s$ is the map $s (a) = (0, a).$ With respect to the above splitting, the maps $i$ and $j$ are the obvious ones. Moreover, the property (\ref{s-property}) implies that $\alpha_E = (\alpha_M, \alpha)$ and $\beta_E = (\beta_M, \beta).$ Since $j \circ \mu_E ((0, a), (0, b)) = \mu (a, b)$ as $j$ is an algebra map, we have $\mu_E ((0, a), (0, b)) = (f (a, b), \mu (a, b))$, for some $f \in C^2_{\mathrm{Hoch}} (A, M).$ The bihom-associativity of $\mu_E$ implies that $f$ is a $2$-cocycle. Similarly, one can observe that any two equivalent extensions are related by a map $E = M \oplus A \xrightarrow{\phi} M \oplus A = E'$, $(m, a) \mapsto (m + g(a), a)$ for some $g \in C^1_{\mathrm{Hoch}} (A, M)$. Since $\phi$ is an algebra morphism, we have
\begin{align*}
\phi \circ \mu_E ((0, a), (0, b)) = \mu'_{E} (\phi (0, a) , \phi (0, b))
\end{align*}
which implies that $f' (a, b) = f (a, b) - (\delta_{\mathrm{Hoch}} g)(a, b)$. Here $f'$ is the $2$-cocycle induced by the extension $E'$. This shows that the map $\mathcal{E}xt (A, M) \rightarrow H^2_{\mathrm{Hoch}} (A, M)$ is well defined. Moreover, these two maps are inverses to each other.
\end{proof}

In the following, we will be mostly interested in the Hochschild cohomology of $A$ with coefficients in itself. We will see that the second Hochschild cohomology $H^2_{\mathrm{Hoch}}(A, A)$ plays an important role in the deformation theory of $A$ (see section \ref{sec-def}).

\section{Gerstenhaber structure on cohomology}
In this section, we show that the cochain complex defining the Hochschild cohomology of a bihom-associative algebra inherits a structure of an operad with a multiplication. Hence the results of Subsection \ref{subsec-operad-mul} imply that the cohomology inherits a Gerstenhaber structure. 

\subsection{Bi-twisted endomorphism operad}

Let $A$ be a vector space and $\alpha, \beta : A \rightarrow A$ be two commuting linear maps. For each $k \geq 1$, we define $C^k_{\alpha, \beta} (A, A)$ to be the space of all multilinear maps $f : A^{\otimes k} \rightarrow A$ satisfying $\alpha \circ f = f \circ \alpha^{\otimes k} $ and $\beta \circ f = f \circ \beta^{\otimes k}.$
We define an operad structure on $ \{ \mathcal{O}(k) |~ k \geq 1 \}$ where $\mathcal{O}(k) = C^k_{\alpha, \beta} (A, A)$, for $k \geq 1$. For $1 \leq i \leq m$,
we define partial compositions $\circ_i : \mathcal{O}(m) \otimes \mathcal{O}(n) \rightarrow \mathcal{O}(m+n-1)$ by
$$(f \circ_i g)(a_1, \ldots, a_{m+n-1}) = f ( \alpha^{n-1}a_1, \ldots, \alpha^{n-1}a_{i-1}, g (a_i, \ldots, a_{i+n-1}), \beta^{n-1} a_{i+n}, \ldots, \beta^{n-1} a_{m+n-1}),$$
for $f \in \mathcal{O}(m),~ g \in \mathcal{O}(n)$ and $a_1, \ldots, a_{m+n-1} \in A$. 

\begin{prop}\label{hom-ass-operad}
	The partial compositions ~$\circ_i$ define a non-$\sum$ operad structure on $C^\bullet_{\alpha, \beta} (A,A)$ with the identity element given by the identity map {\em id} $\in C^1_{\alpha, \beta} (A, A)$.
\end{prop}

\begin{proof}
For $f \in C^m_{\alpha, \beta} (A, A),~ g \in C^n_{\alpha, \beta} (A, A),~ h \in C^p_{\alpha, \beta} (A,A)$ and $1 \leq i \leq m,~ 1 \leq j \leq n$, we have
\begin{align*}
&((f \circ_i g) \circ_{i+j-1} h)(a_1, \ldots, a_{m+n+p-2}) \\
&= (f \circ_i g) \big(\alpha^{p-1} a_1, \ldots, \alpha^{p-1} a_{i+j-2} ,~ h (a_{i+j-1}, \ldots, a_{i+j+p-2}), \ldots, \beta^{p-1} a_{m+n+p-2} \big) \\
&= f \big(  \alpha^{n+p-2}a_1, \ldots, \alpha^{n+p-2} a_{i-1}, g \big( \alpha^{p-1}a_i, \ldots, h (a_{i+j-1}, \ldots, a_{i+j+p-2}), \ldots, \beta^{p-1} a_{i+n+p-2}  \big),\\
& \hspace*{10cm} \ldots, \beta^{n+p-2} a_{m+n+p-2}    \big) \\
&= f \big(  \alpha^{n+p-2}a_1, \ldots, \alpha^{n+p-2} a_{i-1} , (g \circ_j h) (a_i, \ldots, a_{i+n+p-2}) , \ldots, \beta^{n+p-2} a_{m+n+p-2} \big) \\
&= (f \circ_i (g \circ_j h)) (a_1, \ldots, a_{m+n+p-2}).
\end{align*}
Similarly, for $1 \leq i < j \leq m$, we have
\begin{align*}
&((f \circ_i g) \circ_{j+n-1} h ) (a_1, \ldots, a_{m+n+p-2}) \\
&= (f \circ_i g) \big( \alpha^{p-1} a_1, \ldots, \alpha^{p-1} a_{j+n-2}, h (a_{j+n-1}, \ldots, a_{j+n+p-2}), \beta^{p-1} a_{j+n+p-1}, \ldots, \beta^{p-1} a_{m+n+n-2} \big) \\
&= f \big( \alpha^{n+p-2}a_1, \ldots, \alpha^{n+p-2} a_{i-1} , g (\alpha^{p-1} a_i, \ldots, \alpha^{p-1} a_{i+n-1}), \beta^{n-1} \alpha^{p-1} a_{i+n}, \ldots, \beta^{n-1} \alpha^{p-1} a_{j+n-2},\\
&  \hspace*{4cm}    \beta^{n-1} h(a_{j+n-1}, \ldots, a_{j+n+p-2}), \beta^{n+p-2} a_{j+n+p-1}, \ldots, \beta^{n+p-2} a_{m+n+p-2}  \big) \\
&= (f \circ_j h) \big(  \alpha^{n-1} a_1, \ldots, \alpha^{n-1} a_{i-1},~ g (a_i, \ldots, a_{i+n-1}),~ \beta^{n-1} a_{i+n}, \ldots, \beta^{n-1} a_{m+n+p-2}    \big) \\
&= ((f \circ_j h) \circ_i g ) (a_1, \ldots, a_{m+n+p-2}).
\end{align*}
It is also easy to see that the identity map id is the identity element of the operad. Hence, the proof.
\end{proof}

In view of (\ref{eqn-2}), the corresponding compositions
\begin{center}
$\gamma_{\alpha, \beta} : \mathcal{O} (k) \otimes \mathcal{O} (n_1) \otimes \cdots \otimes \mathcal{O} (n_k) \rightarrow \mathcal{O} (n_1 + \cdots + n_k)$
\end{center}
are given by
\begin{align*}
&\gamma_{\alpha, \beta} (f; g_1, \ldots, g_k) (a_1, \ldots, a_{n_1 + \cdots + n_k})\\
&=~ f \big( \alpha^{\sum_{l=2}^{k} |g_l|} g_1 (a_1, \ldots, a_{n_1}), \ldots, ~ \alpha^{\sum_{l>i} |g_l|} ~ \beta^{\sum_{l <i} |g_l|}~ g_i (a_{n_1 + \cdots + n_{i-1} +1}, \ldots, a_{n_1 + \cdots+ n_i}) ,\\
& \qquad \qquad \qquad \qquad \qquad \ldots,~ \beta^{\sum_{l=1}^{k-1} |g_l|} g_k (a_{n_1 + \cdots + n_{k-1} + 1}, \ldots, a_{n_1 + \cdots + n_k}) \big),
\end{align*}
for $f \in \mathcal{O}(k),~ g_i \in \mathcal{O}(n_i)$ and $a_1, \ldots, a_{n_1 + \cdots + n_k} \in A$.

\begin{remark}
When $\alpha = \beta$, one recovers the operad considered in \cite{das2} associated to any vector space $A$ with a linear map $\alpha$. In particular, when $\alpha$ and $\beta$ are both identity map, one gets our favourite endomorphism operad. In future, we plan to study this bi-twisted endomorphism operad in more details.
\end{remark}

Note that, the corresponding degree $-1$ graded Lie bracket on $C^\bullet_{\alpha, \beta} (A, A)$ is given by
\begin{align*}
[f , g] = f \circ g - (-1)^{(m-1)(n-1)} g \circ f,
\end{align*}
where 
\begin{align}\label{circ-operation}
~&(f \circ g)(a_1, \ldots, a_{m+n-1}) := \{f\} \{g \} (a_1, \ldots, a_{m+n-1})  \nonumber \\
~&= \sum_{i=1}^{m} (-1)^{(n-1)(i-1)}~ f (\alpha^{n-1} a_1, \ldots, \alpha^{n-1} a_{i-1}, g (a_i, \ldots, a_{i+n-1}), \beta^{n-1} a_{i+n}, \ldots, \beta^{n-1} a_{m+n-1} ),
\end{align}
for $f \in C^m_{\alpha, \beta} (A, A), ~ g \in C^n_{\alpha, \beta} (A, A)$ and $a_1, \ldots, a_{m+n-1} \in A$.

\subsection{Cohomology as Gerstenhaber algebra}\label{cohomo-bihom}
Let $(A, \mu, \alpha, \beta)$ be a bihom-associative algebra. Consider the operad $C^\bullet_{\alpha, \beta} (A, A)$ as given by Proposition \ref{hom-ass-operad}. Note that $C^\bullet_{\alpha, \beta} (A, A) = C^\bullet_{\mathrm{Hoch}} (A, A)$ the space of Hochschild cochains.  The multiplication $\mu \in C^2_{\alpha, \beta} (A, A)$. Moreover, we have
\begin{align*}
\{\mu\} \{ \mu \} (a,b,c) =~& \gamma_{\alpha, \beta} (\mu; \mu, \text{id}) (a,b,c) - \gamma_{\alpha, \beta} (\mu; \text{id}, \mu) (a,b,c) \\
=~& \mu ( \mu(a,b), \beta (c)) - \mu (\alpha (a), \mu (b,c)) = 0, ~~ \text{ for all } a, b, c \in A.
\end{align*}
Therefore, $\mu$ defines a multiplication on the operad $C^\bullet_{\alpha, \beta} (A,A).$
The corresponding dot product (which we denote by $\cup_{\alpha, \beta}$) on $C^\bullet_{\alpha, \beta} (A, A)$ is given by
\begin{align}\label{bihom-cup}
 (f \cup_{\alpha, \beta} g)(a_1, \ldots, a_{m+n}) = (-1)^{mn} ~\mu (f(\alpha^{n-1}a_1, \ldots, \alpha^{n-1} a_m), g (\beta^{m-1}a_{m+1}, \ldots, \beta^{m-1}a_{m+n})  ),
 \end{align}
for $f \in C^m_{\alpha, \beta} (A, A), ~g \in C^n_{\alpha, \beta} (A,A)$ and $a_1, \ldots, a_{m+n} \in A$. Moreover, the differential (which we denote by $d_{\alpha, \beta}$) is given by
\begin{align}\label{two-differential}
 d_{\alpha, \beta} f := \mu \circ f - (-1)^{|f|} f \circ \mu = (-1)^{|f| + 1}~ \delta_{\mathrm{Hoch}} (f),
 \end{align}
where $\delta_{\mathrm{Hoch}}$ is the Hochschild coboundary map defined in (\ref{bihom-diff}).
The last identity follows from an easy calculation. This in particular shows that $(\delta_{\mathrm{Hoch}})^2 = 0.$ Note that
the cohomology induced by the differentials $d_{\alpha, \beta}$ and $\delta_{\mathrm{Hoch}}$ are isomorphic since these
differentials coincide up to a sign.

\medskip

Therefore, by considering the operations on the Hochschild cochain complex $C^\bullet_{\mathrm{Hoch}} (A, A)$ which are induced by the operad with multiplication, one obtains the following. (Compare with Theorem \ref{gers-voro-thm} and Remark \ref{gers-voro-rem}.)

\begin{thm}
	Let $(A, \mu, \alpha, \beta)$ be a bihom-associative algebra. Then its Hochschild cochain complex $C^\bullet_{\mathrm{Hoch}} (A, A)$ carries a homotopy $G$-algebra structure. Hence, its Hochschild cohomology $H^\bullet_{\mathrm{Hoch}} (A, A)$ inherits a Gerstenhaber algebra structure.
\end{thm}

\begin{remark}
When $\alpha = \beta$, one recovers the Gerstenhaber algebra structure on the Hochschild cohomology of a hom-associative algebra \cite{das2} (see also \cite{das1}). In particular, if $\alpha = \beta = \text{id}$, one gets the classical result that the Hochschild cohomology of an associative algebra inherits a Gerstenhaber algebra structure \cite{gers-voro, gers}.
\end{remark}

\section{Formal deformations}\label{sec-def}
In this section, we study formal $1$-parameter deformation of bihom-associative algebras. The results of this section are similar to classical results about deformation of associative algebras \cite{gers-def}.

Let $A = (A, \mu, \alpha, \beta)$ be a bihom-associative algebra. A $1$-parameter formal deformation of $A$ is defined by a $\mathbb{K}[[t]]$-bilinear map $\mu_t : A [[t]] \times A [[t]] \rightarrow A[[t]]$ of the form 
$\mu_t = \sum_{i \geq 0}  \mu_i t^i$
where each $\mu_i \in C^2_{\mathrm{Hoch}} (A, A)$  with $\mu_0 = \mu$ such that the following holds
\begin{align*}
\mu_t (\alpha(a), \mu_t (b,c)) = \mu_t (\mu_t (a,b), \beta (c)), ~~~~ \text{ for all } a, b, c \in A.
\end{align*}
This is equivalent to a system of equations: for $n \geq 0,$
\begin{align}\label{deformation-eqn}
\sum_{i+j = n}  \bigg(  \mu_i (\alpha (a), \mu_j (b,c))  - \mu_i (\mu_j (a,b), \beta (c)) \bigg) = 0, ~~~ ~~~~ \text{ for all } a, b, c \in A.
\end{align}

\begin{remark}\label{remark-def-reln}
In terms of $\circ$ operations as in (\ref{circ-operation}), these system reads as $\sum_{i+j = n} \mu_i \circ \mu_j = 0$, for all $n \geq 0$.
\end{remark}

For $n = 0$, we get the given bihom-associativity of $\mu$. For $n=1$, we get the following.
\begin{lemma}
If $\mu_t = \sum_{i \geq 0} \mu_i t^i $ is a deformation of the bihom-associative algebra $A$, then $\mu_1$ is a Hochschild $2$-cocycle (i.e., $\delta_{\mathrm{Hoch}} (\mu_1) = 0$).
\end{lemma}

The $2$-cocycle $\mu_1$ is called the infinitesimal of the deformation $\mu_t$. More generally, if $\mu_1 = \cdots = \mu_{n-1} = 0$ and $\mu_n$ is non-zero, then $\mu_n$ is a $2$-cocycle.

\begin{defn}
Two deformations $\mu_t = \sum_{i \geq 0} \mu_i t^i$ and $\mu_t' = \sum_{i \geq 0} \mu_i' t^i$ of the bihom-associative algebra $A$ are said to be equivalent if there exists a formal automorphism $\phi_t : A [[t]] \rightarrow A[[t]]$ of the form $\phi_t = \sum_{i \geq 0} \phi_i t^i$ (where $\phi_i \in C^1_{\mathrm{Hoch}} (A, A)$ with $\phi_0 = \text{id}$) such that
\begin{align*}
\phi_t (\mu_t (a,b)) = \mu_t' (\phi_t (a), \phi_t (b)), ~~~~ \text{ for all } a, b \in A.
\end{align*}
\end{defn}

This condition again leads to a system of equations: for $n \geq 0,$
\begin{align*}
\sum_{i + j = n} \phi_i (\mu_j (a, b)) = \sum_{i+j+ k = n} \mu_i' (\phi_j (a), \phi_k (b)), ~~~~ \text{ for all } a, b \in A.
\end{align*}

In particular, for $n = 1$, we obtain 
\begin{align*}
\mu_1 (a,b) - \mu_1' (a, b) = \mu (a, \phi_1 (b)) - \phi_1 (\mu(a,b)) + \mu (\phi_1 (a), b).
\end{align*}
This shows that infinitesimals corresponding to equivalent deformations are cohomologous and therefore they give rise to a same cohomology class in $H^2_{\mathrm{Hoch}} (A, A).$

\begin{defn}
A bihom-associative algebra $A$ is called rigid if any deformation of $A$ is equivalent to the trivial deformation $\mu_t = \mu$.
\end{defn}

\begin{prop}
Let $\mu_t =  \sum_{i \geq 0}  \mu_i t^i$ be a deformation of $A$ not equivalent to the trivial one. Then $\mu_t$ is equivalent to some deformation
\begin{center}
${\mu}'_t = \mu +  {\mu}'_p t^p +  {\mu}'_{p+1} t^{p+1} + \cdots $
\end{center}
in which the first non-vanishing term ${\mu}'_p$ is not a coboundary.
\end{prop}
The proof is similar to the associative algebra case \cite{gers}. Hence, we do not repeat it here.
As a corollary, we obtain the following.
\begin{thm}\label{2-zero-rigid}
If $H^2_{\mathrm{Hoch}} (A, A) = 0$, then $A$ is rigid.
\end{thm}

Let $A$ be an associative algebra and consider the classical Hochschild complex $(C^\bullet_{\mathrm{HochAss}} (A, A), \delta_{\mathrm{HochAss}}).$ If $\alpha , \beta : A \rightarrow A$ are two commuting algebra morphisms, then $(C^\bullet_{\mathrm{HochAss}, \alpha, \beta} (A, A), \delta_{\mathrm{HochAss}})$ is a subcomplex, where
\begin{align*}
C^n_{\mathrm{HochAss}, \alpha, \beta} (A, A) = \{ f : A^{\otimes n} \rightarrow A |~ \alpha \circ f = f \circ \alpha^{\otimes n} \text{ and } \beta \circ f = f \circ \beta^{\otimes n} \}.
\end{align*}
We denote the cohomology of this subcomplex by $H^\bullet_{\mathrm{HochAss}, \alpha, \beta} (A, A)$.
If $\alpha, \beta$ are invertible and $H^2_{\mathrm{HochAss}, \alpha, \beta} (A, A) = 0$, then we claim that the second Hochschild cohomology of the bihom-associative algebra $A = (A, \mu \circ (\alpha \otimes \beta), \alpha, \beta)$ is zero. For any $f \in C^2_{\text{HochBihom}} (A, A)$, we consider the bilinear map $f \circ (\alpha^{-1} \otimes \beta^{-1})$ on $A$. Then
\begin{align*}
\alpha \circ f \circ (\alpha^{-1} \otimes \beta^{-1}) = f \circ \alpha^{\otimes 2} \circ (\alpha^{-1} \otimes \beta^{-1}) = f \circ (\text{id} \otimes \alpha \circ \beta^{-1}) =~& f \circ ( \text{id} \otimes \beta^{-1} \circ \alpha ) \\=~& f \circ (\alpha^{-1} \otimes \beta^{-1}) \circ \alpha^{\otimes 2}.
\end{align*}
Similarly, we have $\beta \circ f \circ (\alpha^{-1} \otimes \beta^{-1}) = f \circ (\alpha^{-1} \otimes \beta^{-1}) \circ \beta^{\otimes 2}$.
 Hence $f \circ (\alpha^{-1} \otimes \beta^{-1}) \in C^2_{\text{HochAss}, \alpha, \beta} (A, A)$.
 Define a map
\begin{align*}
C^2_{\mathrm{HochBihom}} (A, A) \rightarrow C^2_{\mathrm{HochAss}, \alpha, \beta} (A, A), ~ f \mapsto f \circ (\alpha^{-1} \otimes \beta^{-1}).
\end{align*}
If $f \in C^2_{\mathrm{HochBihom}} (A, A)$ is a $2$-cocycle, then it follows that $f \circ (\alpha^{-1} \otimes \beta^{-1}) \in C^2_{\mathrm{HochAss}, \alpha, \beta} (A, A)$ is a $2$-cocycle.  Since $H^2_{\mathrm{HochAss}, \alpha, \beta} (A, A) = 0$, we have 
$f \circ (\alpha^{-1} \otimes \beta^{-1}) = \delta_{\mathrm{HochAss}} (g)$, for some $g \in C^1_{\mathrm{HochAss}} (A, A)$. In that case, $f = \delta_{\mathrm{HochBihom}} (g)$. Hence the claim follows.
Thus, the bihom-associative algebra $(A, \mu \circ (\alpha \otimes \beta), \alpha, \beta)$ is rigid.

\medskip

Let $A = (A, \mu, \alpha, \beta)$ be a bihom-associative algebra. 
Consider the space $A[[t]]/ (t^{n+1})$ of polynomials in $t$ of degree $\leq n$ with coefficients in $A$. It is a module over $\mathbb{K}[[t]]/ (t^{n+1})$. The linear maps $\alpha, \beta $ induce $\mathbb{K}[[t]]/(t^{n+1})$-linear endomorphisms of $A[[t]]/ (t^{n+1})$ in a natural way.

A deformation of order $n$ consists of a sum $\mu_t = \sum_{i = 0}^{n} \mu_i t^i$ such that $(A[[t]]/ (t^{n+1}), \mu_t, \alpha, \beta )$ is a bihom-associative algebra over $\mathbb{K}[[t]]/ (t^{n+1})$. In the following, we assume that $H^2_{\mathrm{Hoch}} (A, A) \neq 0$ so that one may obtain non-trivial deformations. Next, we consider the problem of extending a deformation of order $n$ to a deformation of order $n+1$.
Suppose there is an element $\mu_{n+1} \in C^2_{\mathrm{Hoch}} (A, A)$ such that
$${\mu}'_t = \mu_t +  \mu_{n+1} t^{n+1}$$
is a deformation of order $n+1$. Then we say that $\mu_t$ extends to a deformation of order $n+1$.

Since we assume that $\mu_t =  \sum_{i = 0}^n  \mu_i t^i$ is a deformation of order $n$, it follows from Remark \ref{remark-def-reln} that
\begin{align}\label{deform-rel}
\mu \circ \mu_i + \mu_1 \circ \mu_{i-1} + \cdots + \mu_{i-1} \circ \mu_1 + \mu_i \circ \mu = 0, ~~~ \text{ for } i = 1, 2, \ldots, n,
\end{align}
or, equivalently, $\delta_{\mathrm{Hoch}} (\mu_i) = - \sum_{p+q = i, p, q \geq 1} \mu_p \circ \mu_q$.
For ${\mu}'_t = \mu_t +  \mu_{n+1} t^{n+1}$ to be a deformation of order $n+1$, one more condition needs to be satisfied, namely,
\begin{center}
$\mu \circ \mu_{n+1} + \mu_1 \circ \mu_{n} + \cdots + \mu_{n} \circ \mu_1 + \mu_{n+1} \circ \mu = 0.$
\end{center}
Hence $\mu_{n+1}$ must satisfy
\begin{center}
$ \delta_{\mathrm{Hoch}} (\mu_{n+1}) = - \sum_{i+j =n+1, i, j \geq 1} \mu_i \circ \mu_{j}.$
\end{center}
The right hand side of the above equation is called the obstruction to extend the deformation $\mu_t$ to a deformation of order $n+1$.

\begin{prop}
The obstruction is a Hochschild $3$-cocycle, that is,
\begin{center}
$ \delta_{\mathrm{Hoch}} \big(- \sum_{i+j =n+1, i, j \geq 1} \mu_i \circ \mu_{j} \big) = 0.$
\end{center}
\end{prop}
\begin{proof}
For any $2$-cochains $f, g \in C^2_{\mathrm{Hoch}} (A, A)$, it is easy to see that
\begin{align*}
\delta_{\mathrm{Hoch}} (f \circ g) = f \circ \delta_{\mathrm{Hoch}} (g)~ -~ \delta_{\mathrm{Hoch}} (f) \circ g ~+~ g \cup_{\alpha, \beta} f ~-~ f \cup_{\alpha, \beta} g.
\end{align*}
(See \cite{gers} for associative case and \cite{das1} for hom-associative case.) Therefore,
\begin{align*}
\delta_{\mathrm{Hoch}} \big(- \sum_{i+j =n+1, i, j \geq 1} \mu_i \circ \mu_{j} \big) =~& - \sum_{i+j =n+1, i, j \geq 1} \big( \mu_i \circ \delta_{\mathrm{Hoch}} (\mu_j) - \delta_{\mathrm{Hoch}} (\mu_i) \circ \mu_j \big) \\
=~& \sum_{p+q+r = n+1, p, q, r \geq 1} \big(  \mu_p \circ (\mu_q \circ \mu_r) - (\mu_p \circ \mu_q) \circ \mu_r \big)\\ =~& \sum_{p+q+r = n+1, p, q, r \geq 1} A_{p, q, r}    \qquad \text{(say)}.
\end{align*}
The product $\circ$ is not associative, however, it satisfies the pre-Lie identity (\ref{pre-lie-iden}). This in particular implies that $A_{p, q, r} = 0$ whenever $q = r$. Finally, if $q \neq r$ then $A_{p, q, r} + A_{p, r, q} = 0$ by the pre-Lie identity (\ref{pre-lie-iden}). Hence we have $\sum_{p+q+r = n+1, p, q, r \geq 1} A_{p, q, r}   = 0$.
%
%
%
\end{proof}

It follows from the above proposition that the obstruction defines a cohomology class in $H^3_{\mathrm{Hoch}} (A, A)$. If this cohomology class is zero, then the obstruction is given by a coboundary (say $\delta_{\mathrm{Hoch}} (\mu_{n+1})$). In other words, ${\mu}'_t = \mu_t +  \mu_{n+1} t^{n+1}$ defines a deformation of order $n+1$.

As a summary, we get the following.
\begin{thm}\label{3-zero-extension}
If $H^3_{\mathrm{Hoch}} (A, A) = 0$, every deformation of order $n$ can be extended to a deformation of order $n+1$.
\end{thm}

\vspace{0.2cm}

\section{Bihom-$A_\infty$-algebras}
In this section, we introduce bihom-associative algebras up to homotopy (bihom-$A_\infty$-algebras in short) and study some particular classes of these homotopy algebras. Our definition of a  bihom-$A_\infty$-algebra is purely motivated by the classical definition of an $A_\infty$-algebra introduced by Stasheff \cite{sta} and the Hochschild cohomology of bihom-associative algebras introduced in section \ref{sec-cohomology}. The last assertion will be clear in Theorems \ref{skeletal-2} and \ref{skeletal-n}.

 \begin{defn}\label{bihom-a-inf}
  A bihom-$A_\infty$-algebra consists of a graded vector space $A = \oplus A_i$ together with
	\begin{itemize}
	\item[(i)] a collection $\{ m_k | ~ 1 \leq k < \infty \}$ of linear maps
		$m_k : A^{\otimes k} \rightarrow A$ with deg~$(m_k) = k-2,$
	\item[(ii)] two commuting linear maps $\alpha, \beta : A \rightarrow A$ of degree $0$ with 
	\begin{center}
	$\alpha \circ m_k = m_k \circ \alpha^{\otimes k}$ \hspace{0.2cm} and \hspace{0.2cm} $\beta \circ m_k = m_k \circ \beta^{\otimes k}$ ~~~for all $k \geq 1$
	\end{center}
	\end{itemize}
such that for all $n \geq 1$,
\begin{align}\label{ha-eqn}
\sum_{i+j = n+1}^{} \sum_{\lambda =1}^{j} (-1)^{\lambda (i+1) + i (|a_1| + \cdots + |a_{\lambda -1 }|)} ~ m_{j} \big(  \alpha^{i-1}a_1, \ldots, \alpha^{i-1} a_{\lambda -1}, m_i ( a_{\lambda}, \ldots, a_{\lambda + i-1}),\\
 \beta^{i-1} a_{\lambda + i}, \ldots, \beta^{i-1} a_n   \big) = 0 \nonumber,
\end{align}
for $a_i \in A_{|a_i|}, ~ 1 \leq i \leq n$.
\end{defn}

It follows from the above conditions that $m_1$ is a differential and it is a graded derivation for the product $m_2$. The product $m_2$ is graded bihom-associative modulo some terms involving $m_3$, and so on. 

A dg bihom-associative algebra is a bihom-$A_\infty$-algebra with $m_k = 0$ for $k \geq 3$. Given a dg associative algebra and two commuting dg associative algebra morphisms, one can construct dg bihom-associative algebra by Yau twist.

A bihom-$A_\infty$-algebra as above is denoted by $(A, m_k, \alpha, \beta)$. When $\alpha = \beta$, one recovers $HA_\infty$-algebras recently introduced in \cite{das3}. In particular, when both of them are identity map, one get $A_\infty$-algebras. Thus, they are examples of bihom-$A_\infty$-algebras. The direct sum of two bihom-$A_\infty$-algebras is again a bihom-$A_\infty$-algebra.

An $n$-term bihom-$A_\infty$-algebra is a bihom-$A_\infty$-algebra $(A, m_k, \alpha, \beta)$ whose underlying graded vector space $A$ is concentrated in degrees from $0$ to $n-1$.
In this case, it is easy to see that $m_k = 0$, for $k > n+1.$ Thus, a $2$-term bihom-$A_\infty$-algebra is a tuple $(A_1 \xrightarrow{d} A_0, m_2, m_3, \alpha_0, \beta_0, \alpha_1 \beta_1)$ consisting of a chain complex $A:= (A_1 \xrightarrow{d} A_0)$, maps $m_2 : A_i \otimes A_j \rightarrow A_{i+j}$ and $m_3 : A_0^{\otimes 3} \rightarrow A_1$, two commuting chain maps $\alpha = (\alpha_0, \alpha_1)$ and $\beta = (\beta_0, \beta_1)$ from $A$ to $A$ satisfying a set of relations that follow from Definition \ref{bihom-a-inf}. The relations are explicitly written in \cite{das3} for the case of a $2$-term $HA_\infty$-algebra (i.e. when $\alpha = \beta$). One should get an idea from the above mentioned reference.

The results of the next two subsections are generalization of some standard results of Baez and Crans \cite{baez-crans}. In particular, our discussion is analogous to the results of \cite{das3} about $2$-term $HA_\infty$-algebras. Thus, most of the time, we only give a sketch of a proof and refer to \cite{das3}.

\subsection{Skeletal algebras}\label{subsec-skeletal}
\begin{defn}
	A $2$-term bihom-$A_\infty$-algebra $(A_1 \xrightarrow{d} A_0, m_2, m_3, \alpha_0,  \beta_0, \alpha_1, \beta_1)$ is called skeletal if $d = 0$. 
\end{defn}

Let $(A_1 \xrightarrow{0} A_0, m_2, m_3, \alpha_0, \beta_0, \alpha_1, \beta_1)$ be a skeletal $2$-term bihom-$A_\infty$-algebra. Then it follows from the arguments similar to \cite{das3} that $(A_0, m_2, \alpha_0, \beta_0)$ is a bihom-associative algebra. Moreover, the maps
\begin{align*}
m_2 : A_0 \otimes A_1 \rightarrow A_1, ~~ (a,m) \mapsto m_2 (a,m) \\
m_2 : A_1 \otimes A_0 \rightarrow A_1, ~~ (m,a) \mapsto m_2 (m, a)
\end{align*}
define a bimodule structure on $A_1$ with respect to commuting linear maps $\alpha_1, \beta_1 : A_1 \rightarrow A_1$.


\begin{thm}\label{skeletal-2}
	There is a one-to-one correspondence between skeletal $2$-term bihom-$A_\infty$-algebras and tuples $( (A, m, \alpha, \beta), M, \cdot,  \alpha_M, \beta_M , \theta)$, where $A=(A, m, \alpha, \beta)$ is a bihom-associative algebra, $(M, \cdot, \alpha_M, \beta_M)$ is a bimodule over $A$ and $\theta$ is a $3$-cocycle of the bihom-associative algebra $A$ with coefficients in the bimodule $M$.
%
\end{thm}

\begin{proof}
	Let $(A_1 \xrightarrow{0} A_0, m_2, m_3, \alpha_0,  \beta_0, \alpha_1, \beta_1)$ be a skeletal $2$-term bihom-$A_\infty$-algebra. 
Then it follows that
	\begin{center}
	$ m_3 : A_0 \otimes A_0 \otimes A_0 \rightarrow A_1$
	\end{center}
	defines a $3$-cocycle of the bihom-associative algebra $(A_0, m_2, \alpha_0, \beta_0)$ with coefficients in the bimodule $(A_1, m_2, \alpha_1, \beta_1)$.
	
	Conversely, given a tuple $((A, m, \alpha, \beta), M, \cdot, \alpha_M, \beta_M, \theta)$ as in the statement, define $A_0 = A$, $\alpha_0 = \alpha$, $\beta_0 = \beta$ and $A_1= M$, $\alpha_1 = \alpha_M, \beta_1 = \beta_M$. We define multiplications $m_2 : A_i \otimes A_j \rightarrow A_{i+j}$ and $m_3 : A_0 \otimes A_0 \otimes A_0 \rightarrow A_{1}$ by
	\begin{align*}
	m_2 (a,b) = m (a,b),~~~
	m_2 (a,m) = a \cdot m,~~~
	m_2 (m,a) = m \cdot a ~~ \text{ and } ~~
	m_3 (a,b,c) = \theta (a,b,c),
	\end{align*}
	for  $a, b, c \in A_0 = A$ and $m \in A_1 = M.$ Then $(A_1 \xrightarrow{0} A_0, m_2, m_3, \alpha_0, \beta_0,  \alpha_1, \beta_1)$ is a skeletal $2$-term bihom-$A_\infty$-algebra.
\end{proof}

More generally, we can consider those $n$-term bihom-$A_\infty$-algebras whose underlying chain complexes are of the form
\begin{align}\label{n-term-com}
( A_{n-1} \xrightarrow{0} 0 \xrightarrow{0} \cdots \xrightarrow{ 0} 0 \xrightarrow{0} A_0 ).
\end{align}
That is, non-zero terms are concentrated in degrees $0, n-1$ and the differential is zero. The degree $0$ linear maps $\alpha, \beta$ consist of components $\alpha_0 , \beta_0 : A_0 \rightarrow A_0$ and $\alpha_{n-1}, \beta_{n-1} : A_{n-1} \rightarrow A_{n-1}$. Moreover, the multiplications $m_k$'s are non-zero only for $k=2, n+1$. It follows from Definition \ref{bihom-a-inf} that $A = (A_0, m_2, \alpha_0, \beta_0)$ forms a bihom-associative algebra and the maps $m_2 : A_0 \otimes A_{n-1} \rightarrow A_{n-1}$, $m_2 : A_{n-1} \otimes A_0 \rightarrow A_{n-1}$ define a bimodule structure on $(A_{n-1}, \alpha_{n-1}, \beta_{n-1}).$


Then we have the following. The proof is similar to Theorem \ref{skeletal-2} and hence we do not repeat it here.

\begin{thm}\label{skeletal-n}
	There is a one-to-one correspondence between $n$-term bihom-$A_\infty$-algebras whose underlying chain complexes are of the form (\ref{n-term-com}) and tuples $((A, m, \alpha, \beta), M, \cdot, \alpha_M, \beta_M, \theta)$, where $(A, m, \alpha, \beta)$ is a bihom-associative algebra, $(M, \cdot, \alpha_M, \beta_M)$ is a bimodule over $A$ and $\theta$ is an $(n+1)$-cocycle of $A$ with coefficients in the bimodule $M$.
%
\end{thm}

\subsection{Strict algebras}\label{subsec-strict}

\begin{defn}
	A $2$-term bihom-$A_\infty$-algebra $(A_1 \xrightarrow{d} A_0, m_2, m_3, \alpha_0,  \beta_0, \alpha_1, \beta_1)$ is called strict if $m_3 = 0$.
\end{defn}

\begin{exam}\label{exam-strict}
Let $(A, m, \alpha, \beta)$ be a bihom-associative algebra. Take $A_0 = A_1 = A$, $d =~$id,~$m_2 = m$, $\alpha_0 = \alpha_1 = \alpha$ and $\beta_0 = \beta_1 = \beta$. Then $(A_1 \xrightarrow{d} A_0, m_2, m_3 = 0, \alpha_0,  \beta_0, \alpha_1, \beta_1)$ is a strict $2$-term bihom-$A_\infty$-algebra.
\end{exam}

In the following, we introduce crossed module of bihom-associative algebras and show that they are related to strict $2$-term bihom-$A_\infty$-algebras.
\begin{defn}
	A crossed module of bihom-associative algebras consists of a quadruple 
	\begin{center}
	$((A, m_A, \alpha_A, \beta_A), (B, m_B, \alpha_B, \beta_B), dt, \phi)$ 
	\end{center}
	where $(A, m_A, \alpha_A, \beta_A)$ and $(B, m_B, \alpha_B, \beta_B)$ are bihom-associative algebras, $dt : A \rightarrow B$ is a bihom-associative algebra morphism and
	\begin{align*}
	\phi : B \otimes A \rightarrow A,& ~ (b, m) \mapsto \phi (b ,m), \\
	\phi : A \otimes B \rightarrow A,& ~ (m, b) \mapsto \phi(m, b)
	\end{align*}
	defines a $B$-bimodule structure on $A$ with respect to linear maps $\alpha_A$ and $\beta_A$, such that for $ m, n \in A,~ b \in B$,
	\begin{align*}
	dt ( \phi (b, m)) =~& m_B (b, dt(m)),\\
	dt ( \phi(m, b)) = ~& m_B (dt (m), b  ), \\
	\phi ({dt(m)} , n) =~& m_A (m, n),\\
	\phi(m, dt(n)) =~& m_A (m, n),\\
	\phi (  \alpha_B (b), m_A (m, n)) =~& m_A (\phi(b, m), \beta_A (n)),\\
	\phi (  m_A (m,n), \beta_B (b)) =~& m_A (\alpha_A (m), \phi(n,b)).
	\end{align*}
\end{defn}

When $\alpha_A = \beta_A$ and $\alpha_B = \beta_B$, one gets the crossed module of hom-associative algebras \cite{das3}. In particular, when they are identity, one gets the crossed module of associative algebras.

\begin{thm}\label{strict-crossed-mod}
	There is a one-to-one correspondence between strict $2$-term bihom-$A_\infty$-algebras and crossed modules of bihom-associative algebras.
\end{thm}
\begin{proof}
	Let $(A_1 \xrightarrow{d} A_0, m_2 , m_3 = 0, \alpha_0,  \beta_0, \alpha_1, \beta_1)$ be a strict $2$-term $HA_\infty$-algebra. Take $A = A_1$ with the multiplication $m_A$ defined by $m_A (m,n) := m_2 (dm, n) = m_2 (m, dn)$, for $m, n \in A_1$, and linear transformations $\alpha_A = \alpha_1$, $\beta_A = \beta_1$.
Then similar to \cite{das3} one shows that $(A, m_A, \alpha_A, \beta_A)$ is a bihom-associative algebra. Moreover, if we take $B = A_0$ with the multiplication $m_B = m_2 : A_0 \otimes A_0 \rightarrow A_0$ and linear transformations $\alpha_B = \alpha_0$, $\beta_B = \beta_0$, then $(B, m_B, \alpha_B, \beta_B)$ is a bihom-associative algebra. Finally, we define $dt = d : A_1 \rightarrow A_0$ and
	\begin{align*}
	\phi : B \otimes A \rightarrow A, ~ (b, m) & \mapsto  \phi (b,m) = m_2 (b, m), \\
	\phi : A \otimes B \rightarrow A, ~ (m, b) & \mapsto \phi(m,b) = m_2 (m , b),  \text{ for } b \in B, m \in A.
	\end{align*}
	Then one can verify that $(A, B, dt, \phi)$ is a crossed module of bihom-associative algebras.
	
	Converse part is similar. (See \cite{das3} for instance.) Hence the proof.
\end{proof}

\begin{remark}
Note that the crossed module corresponding to the strict $2$-term bihom-$A_\infty$-algebra of Example \ref{exam-strict} is given by $((A, m, \alpha, \beta), (A, m, \alpha, \beta), \text{id}, \mu)$.
\end{remark}

\section{Concluding remarks}
We end this paper with following remarks or questions.

\begin{itemize}
\item[(1)] Hom-associative algebras (thus associative algebras) are special case of bihom-associative algebras. In \cite{hass-shap-sut} the authors consider a definition of Hochschild homology for hom-associative algebras. In the case of bihom-associative algebras, the situation is not that much transparent due to the presence of two commuting maps. We could not find such a homology for bihom-associative algebras. However, it would be interesting if one can find this.

\item[(2)] In this paper, we show that the third Hochschild cohomology of bihom-associative algebras are related to skeletal algebras. In the classical case, it is also known that the third Hochschild cohomology of an associative algebra classify $2$-fold crossed extensions (see \cite{loday}). It is natural to ask the similar question for Hochschild cohomology of bihom-associative algebras.

\item[(3)] We hope that the method of twisting $\alpha$ and $\beta$ in the operad $C^\bullet_{\alpha, \beta} (A, A)$ may help to construct cohomology theories
for other types of Loday algebras (e.g., dialgebra, dendriform algebra, associative trialgebra) twisted by two commuting morphisms that may come up in future. See \cite{yau, das4} for cohomology of various Loday algebras. The existence of homotopy $G$-algebra structure on their cochain complex can be easily shown by the method of \cite{yau} and the above twist.
\end{itemize}

\noindent {\bf Acknowledgements.} The author would like to thank both the referees for their valuable comments on the earlier version of the manuscript that have improved the exposition.
He also wishes to thank Indian Institute of Technology (IIT) Kanpur for financial support.

\end{document}